\documentclass[12pt]{article}
\usepackage{amsmath,amssymb,latexsym}
\usepackage{epsfig}
\usepackage{graphicx}
\usepackage{amsthm}
\usepackage{amsfonts}
\newcommand{\ignore}[1]{}

\def\Bin{{\rm Bin}}

\def\KS2{{\cal KS}2}

\newcommand{\brac}[1]{\left(#1\right)}
\newcommand{\bfrac}[2]{\brac{\frac{#1}{#2}}}

\newcommand{\set}[1]{\left\{#1\right\}}

\def\cD{{\cal D}}

\def\cB{{\cal B}}

\def\cE{{\cal E}}

\def\e{\varepsilon}

\def\cG{{\cal G}}
\def\cX{{\cal X}}
\def\cE{{\cal E}}
\def\cF{{\cal F}}

\def\a{\alpha}
\def\b{\beta}
\def\d{\delta}
\def\D{\Delta}

\def\g{\gamma}
\def\G{\Gamma}
\def\k{\kappa}

\def\z{\zeta}

\def\l{\lambda}

\def\n{\nu}

\def\r{\rho}

\def\S{\Sigma}
\def\t{\tau}

\def\OM{\Omega}

\def\cG{{\cal G}}
\def\Pr{\mbox{{\bf Pr}}}

\def\whp{{\bf whp}}

\def\from{\leftarrow}

\newtheorem{theorem}{Theorem}
\newtheorem{lemma}[theorem]{Lemma}

\newtheorem{claim}[theorem]{Claim}
\newtheorem*{conj}{Conjecture}

\newcommand{\proofend}{\hspace*{\fill}\mbox{$\Box$}}

\newcommand{\card}[1]{\left|#1\right|}

\newcommand{\rdup}[1]{\left\lceil #1 \right\rceil }

\newcommand{\pee}{{\mathcal P}(F)}
\newcommand{\eee}{{\rm END}}
\newcommand{\fee}{{\mathcal F}}

\begin{document}
\title{Hamilton cycles in 3-out}
\author{Tom Bohman\thanks{Supported in part by NSF grants DMS-0401147 and DMS-0701183.}
 and Alan Frieze\thanks{Supported in part by NSF grant CCF-0502793.}
\\Department of Mathematical Sciences,\\ Carnegie Mellon
University,\\ Pittsburgh PA 15213,\\ USA. }

\maketitle

\begin{abstract}
Let \( G_{\rm 3-out} \) denote the random graph on vertex set \([n] \) in which 
each vertex chooses 3 neighbors uniformly at random.  Note that \( G_{\rm 3-out} \)
has minimum degree 3 and average degree 6.  We prove that the probability that
\( G_{\rm 3-out} \) is Hamiltonian goes to 1 as \(n\) tends to infinity.
\end{abstract}

\section{Introduction}
One of the natural questions for a sparse random graph model is 
whether or not a random instance is Hamiltonian \whp. 
The Hamiltonicity threshold for the
basic
models $G_{n,m}$ and $G_{n,p}$ was
established quite precisely by Koml\'os and Szemer\'edi \cite{KS}. See
Bollob\'as \cite{BoHam}, Ajtai, Koml\'os and Szemer\'edi
\cite{AKS} and Bollob\'as and Frieze \cite{BF1} for refinements. 
If $m=\frac{n}{2}(\log n+\log\log n+c)$
then
$$\lim_{n\to\infty}\Pr(G_{n,m}\mbox{ is
Hamiltonian})=\lim_{n\to\infty}\Pr(\d(G_{n,m})\geq 2)=e^{-e^{-c}}$$
where $\d$ denotes minimum degree.
The lesson here is
that the minimum degree seems to be the most important factor in
determining the likelihood of Hamiltonicity in a random graph. This 
naturally leads to
the consideration of models where the minimum degree condition is
automatically satisfied. If we simply condition on
$G_{n,m}$ having minimum degree 2 then about $\frac{1}{6}n\log n$
random edges are needed (see Bollob\'as, Fenner and Frieze
\cite{BFF1}).  The barrier to Hamiltonicity in this model 
is the existence of three vertices of degree 2 that have a common
neighbor.  So, we expect even sparser random graph models with minimum
degree at least three to have Hamilton cycles \whp\ (note that if
we have a linear number of vertices of degree 2 it may be very likely
that there are three vertices of degree 2 that have a common neighbor).
We focus our attention here on random graphs with only a linear 
number of edges and minimum degree 3.
Bollob\'as, Cooper, Fenner and
Frieze \cite{BFF1} showed that if we condition on $G_{n,m}$ having minimum
degree 3 then $m = 128 n$ random edges suffice to give a Hamilton
cycle \whp.  We believe that this result holds so long as the average degree
is at least 3:
\begin{conj}
For any $c \ge 3/2$ the random
graph \( G_{n,cn} \) conditioned on minimum degree 3 has a Hamilton cycle
with high probability.
\end{conj}
\noindent
Another well-studied sparse random graph is the 
random regular graph. Let $G_r$ denote a graph chosen
uniformly at random from the set of $r$-regular graphs with vertex set
$[n]$. Robinson and Wormald \cite{RW} showed $G_r$ is Hamiltonian
\whp\ for $r\geq 3$, $r$ constant. Allowing $r$ to grow with $n$ 
presented some
challenges, but they have now been resolved (see Cooper, Frieze 
and Reed
\cite{CFR} and Krivelevich, Sudakov, Vu and Wormald \cite{KSVW}).

We now come to the topic of this paper: $G_{m\text{--out}}$. We begin 
with vertex set $V=[n]$. Each $v\in V$ independently chooses $m$
random out-neighbors to create the random digraph $D_{m\text{--out}}$. We
then obtain $G_{m\text{--out}}$ by ignoring orientation. 
Note that \( G_{m\text{--out}} \) is a graph with minimum degree \(m\) and average
degree \(2m \).  Note further that
there is the 
potential for multiple edges, but the expected number is $O(m^2)$.
So, we can either allow these multiple edges or condition on them
not occurring.  Since the probability that there are no multiple
edges is bounded away from zero, any property that holds \whp\
in the model that allows multiple edges also holds \whp\ when 
we condition
on no multiple edges. The Hamiltonicity of
$G_{m\text{--out}}$ was first discussed by Fenner and Frieze \cite{FF}. They
showed that $G_{\rm 23-out}$ is Hamiltonian \whp.  
This was improved to \( G_{\rm 10-out} \) by Frieze \cite{F} and to
\( G_{\rm 5-out} \) by
Frieze and {\L}uczak \cite{FL}. Cooper and Frieze \cite{CF} 
showed that 
the digraph
\( D_{\rm 2-in, 2-out} \) (here each vertex chooses 2 out-neighbors 
and 2 in-neighbors
uniformly at random) has a directed Hamilton cycle \whp.  This implies
that \(G_{\rm 4-out} \) is Hamiltonian \whp.  So the main open question is 
whether $G_{\rm 3-out}$ is Hamiltonian \whp\ or not (note that, \whp\
\( G_{\rm 2-out} \) contains 3 vertices of degree 2 that have
a common neighbor and is therefore not Hamiltonian \whp). 
We settle this question:
\begin{theorem}\label{th1}
$$\lim_{n\to\infty}\Pr(G_{\rm 3-out} \text{ is Hamiltonian})=1.$$
\end{theorem}
We give a proof of Theorem~\ref{th1} in Section~\ref{Proof}.  The proof
hinges on a few key Lemmas, the proofs of which are given in later sections.

\section{Proof of Theorem~\ref{th1}}\label{Proof}

The proof has two main steps.  
Before the first step, we set aside a small set of edges that will
not be used until the second step.  In
the 
first step we find a collection of vertex-disjoint
paths and cycles that covers the vertex set.  The set-aside
edges are then used in the second step to transform
this collection of cycles and paths into a Hamilton cycle.

Of course, we need
some notation.
For a digraph $D=(V,A)$ and $v\in V$ let $d^+_D(v),d^-_D(v)$ denote
the out-degree and in-degree, respectively, of $v$ in $D$. 
Let 
$V=[n]$ and let $\OM$ denote the set of digraphs with vertex set $V$
in which $d^+_D(v)=3$ for $v\in V$. 
It will be convenient to let 
the out-arcs for each vertex be chosen 
with replacement and loops allowed. 
Thus $|\OM|= n^{3n}$ and
$D_{3-out}$ is sampled uniformly from $\OM$.  
For a digraph $D$ let $\G_D$ be the underlying graph obtained by
ignoring orientation. Thus $G_{3-out}=\G_{D_{3-out}}$. 

We separate the 
edge set as follows.  We begin by randomly selecting a set of vertices
\(K\) and setting aside the third out-arc of each vertex in \(K \).  
Formally, for 
$K\subseteq V$ define $\OM_K$ to be the set of digraphs $D$
for which $d_D^+(v)=2$ for all $v\in K$ and $d_D^+(v)=3$ for all $v\in V\setminus K$.
Thus, \( | \OM_K | = n^{3(n-k) + 2k} = n^{3n-k} \).
Further, set
\[ \OM_k=\bigcup_{K \in \binom{[n]}{k}}\OM_K.\] 
We work with the uniform distribution on 
\( \Omega_k \times [n]^k \).  The first component gives all but \( k \) of the arcs
in \(D_{3-out} \) and the second component determines the heads of the 
remaining \(k\) arcs.  
For $D\in \OM_k$ we let 
\begin{eqnarray*}
K(D)&=&\{v\in V:\;d_D^+(v)=2\}\\
L(D)& = &\{ v \in K(D) : d^-_D(v) = 0 \}.
\end{eqnarray*}  
We add
the third out-arc of each vertex \( v \in L(D) \) back into the graph 
(thereby arriving at graph with minimum
degree 3).  The resulting digraph is used for the first part of the proof.  The remaining
set-aside arcs (i.e. the third out-arc from each 
\( v \in K \setminus L \)) are used for the second
part of the argument. 

Set 
\[ k = \rdup{\frac{ n }{ \log^{1/2} n}}. \]
Let \(D_1 = (V,A_1)\) denote the digraph drawn uniformly at random from
\( \Omega_k \).  Let \(A\) be the set of arcs then determined by 
the randomly chosen \(k\)-tuple in \( [n]^k \).
Let \(A_2 \) be the set of arcs in \(A\) with tails in \(L(D_1)\) and set \( A_3 = A \setminus A_2\).
Let \(E_3 \) be the set of edges given by ignoring the orientations of the arcs in \(A_3 \). 
Set \( D_2 = D_1 + A_2 \) and \( G_2 = \G_{D_2} \).  In the first part of the 
argument we find a collection of paths and cycles in
\(G_2 \)
that covers the vertex set.  The edges in \(E_3 \) are then 
used in the second part of the argument to convert this 
collection of paths and cycles into a Hamilton
cycle.
%
%
%
Since
$G_2$ is a super-graph of $G_{2-out}$ we have the following:
\begin{lemma}[Fenner, Frieze \cite{FF0}]
\label{lem:connected}
$G_2$ is connected \whp.
\end{lemma}
\noindent
Define a {\bf simple 2-matching} in a graph \(H\)
to be a subgraph of \(H\) with maximum degree 2, i.e. a collection of vertex-disjoint
paths and cycles.  We assume that the vertex set of a simple 2-matching in \(H\) is \( V(H) \).  So some
paths in a 2-matching could be isolated vertices. 
If $F$ is a 2-matching let $\k(F)$ be the number of components in \(F\) and
let $\r(F)$ be the number of path components in \(F\).  
If \(F\) is a simple 2-matching in \(H\) and \( x \in V(H) \) then we
define \( F^x \) to the component in \(F\) that contains the vertex \(x\).  
Finally, if \(F\) is a simple 2-matching, we let \( \pee \) denote 
the set of vertices covered by the path components of \(F\).
The first step in the proof is to establish that \(G_2\) has a simple 2-matching 
with few components.
\begin{lemma}
\label{lem:2match}
{\bf Whp} $ G_2 $ contains a simple 2-matching $F_2$ with $ \k(F_2) \leq  \frac{6n}{ \log n} $.
\end{lemma}
\noindent
Lemma~\ref{lem:2match}
is proved in Section~\ref{sec:2matching}.  The proof uses an extension of 
the Tutte 1-factor theorem.

The {\bf Extension-Rotation Procedure} defined 
in Figure~\ref{fig:ER} 
converts $F_2$ into a Hamilton
cycle.  This
procedure takes as
input a fixed graph \( G_2 = ([n],E) \), the set of arcs \(A_3 \) 
(which can also be viewed as fixed)
and a 2-matching $F = F_2$ in \( G_2\), and, if successful, it 
produces a Hamilton cycle.

\begin{figure}
\caption
{\bf Extension Rotation Procedure }
\label{fig:ER}

\begin{tabbing}

While \( \k(F) + \r(F) > 1 \): \\[7mm]
\hskip1cm \= If  \( \pee = \emptyset \): \\[2mm]
\> \hskip1cm \=
Since \( G_2 \) is connected, there exists an edge \( \{x,y\} \in E \) 
such that \\  
\> \> \ \ \ \( F^x \) and \( F^y \) are distinct cycle components.\\[2mm]
\> \>  Replace \( F^x \) and \( F^y\) with a single path using the edge \( \{x,y\} \).\\[7mm]
\> If  \( \pee \neq \emptyset \):\\[2mm]
\> \> Let \(P\) be a longest path in \(F\). \\[2mm]  
\> \> Let \(X\) be the set of vertices spanned by \(P\). \\[2mm]
\> \> Let \eee\ be the set of end-points of paths that span \(X\). \\[2mm]
\> \> If there exists \( \{x,y\} \in E \) such that
\(x \in \eee \) and \( y \not\in \pee \):\\[2mm]
\> \> \hskip1cm \= Replace \(P\) and \(F^y\) with a single path.\\[2mm]
\> \> If there exists \( \{x,y\} \in E \) such that
\(x \in \eee \) and \( y \in \pee \setminus X \):\\[2mm]
\> \> \> Replace \(P\) and \(F^y\) with two paths by extending a path\\
\> \> \> \ \ \ on \(X\) ending with \(x\) along \( F^y\) to one of the 
endpoints of \(F^y\).\\[2mm]
\> \> If \( N(x) \subseteq X \) for all \( x \in \eee \):\\[2mm]
\> \> \>  Choose an ordering of the elements of \( \eee\) uniformly at random.\\[2mm]
\> \> \>  Consider the elements of \( \eee \) one at a time.  For 
each such \(x \in \eee\):\\[2mm]
\> \> \> \hskip1cm \= Let \( \eee(x) \) be the set of 
endpoints of paths spanning \(X\) \\
\> \> \> \> \ \ \ such that \(x\) is the other endpoint.\\[2mm]
\> \> \> \>
Check whether or not \(x \in K \setminus L \).\\[2mm]  
\> \> \> \> If \(x \in K \setminus L \) determine the arc \( (x,y_x) \in A_3 \). \\[2mm]
\> \> \> \> If \(x \in K \setminus L \) and \( y_x \in \eee(x) \) then use \( \{x,y_x\} \) to
replace \(P\) with\\
\> \> \> \> \ \ \  a single cycle, add \( \{x,y_x\} \) to \(E\), and move onto the next\\
\> \> \> \> \ \ \  iteration of the procedure.\\[2mm]
\> \> \> \> If \(x \not\in K \setminus L \) or \( y_x \not\in \eee(x) \) then move onto 
the next element of \(\eee\). \\[2mm]
\> \> \> If \( x \not\in K \setminus L \) or \( y_x \not \in \eee(x) \) for all \(x \in \eee \) 
then {\bf Fail.}
\end{tabbing}

\end{figure}

For the purposes of analysis we assume that, 
while $G_2$ is given, $K\setminus L$ and the arcs $A_3$ are not specified ab initio. 
At each stage there is a set of possibilities for the parts of $K\setminus L$
and $A_3$ that have not been revealed. The algorithm ``sees'' $G_2$ 
but does not know the orientations of the edges in $D_2$. 
It ``learns'' about $K\setminus L$ and $A_3$ as needed.
More precisely, we condition on $G_2$ and there is a probability distribution for $K\setminus L$.
Information needed
to answer the question ``is $x\in K\setminus L$'' is revealed as the algorithm proceeds.  Similarly,
the random arcs in \(A_3\) are revealed one at a time as necessary as the algorithm proceeds (note that the
heads of these arcs are uniformly distributed on \([n]\)).
Thus we follow the ``principle of deferred decisions'' in a way made precise below.

Appealing to Lemmas~\ref{lem:connected}~and~\ref{lem:2match},
we assume \( G_2 \) is connected and \( \k(F_2) \le 6n/ \log n  \).
In the course of the procedure 
we iteratively modify the
2-matching \(F\) and occasionally add edges from \( E_3 \) to the edge
set \(E\).  In each iteration we attempt to reduce \( \k(F) \),  extend 
the longest path in \(F\) while keeping \(\k(F)\) and \( \r(F) \) fixed, or 
replace a path component in \(F\) with a cycle
component (thereby reducing \( \r(F) \)).  We use edges from \( E \) whenever 
possible.  We resort to the
arcs in \( A_3\) only when we cannot perform any of the operations using only the
edges in \( E \) and
only to replace a path component with a cycle component.  Since \( \r(F) \) only increases in situations
where \( \r(F) = 0 \), and 
in this case we increase \( \r(F) \) to 1 and simultaneously
reduce \( \k(F) \), it follows that the number of arcs from \( A_3 \) that we use
is bounded by \( \k(F) + \r(F) \), which we assume is at most \( 12 n/ \log n  \).

In order to analyze this procedure, we
condition.  For a graph \(G\) we set 
\[  \Omega_k^G = \{ D \in \Omega_k : \G_D = G \} . \]
We prove that the Extension-Rotation Procedure succeeds \whp, and thereby
prove Theorem~\ref{th1}, by working with the partition \( \fee \) 
of \( \Omega_k \times [n]^k \) in which \( (D_1,A) \) and \( (D_1^\prime,A^\prime) \) are in the same
part if  \( \G_{D_1} = \G_{D^\prime_1} \) and
\( \G_{D_2} = \G_{D^\prime_2} \) (in other words, each part is a subset of a set of the form 
\( \Omega_k^G \times [n]^k \) for some graph \(G\)).  One virtue of this
partition is that if we condition on \( (D_1 ,A) \in \Omega_k^G \times [n]^k \) then \(K \setminus L\) is
not determined, it can be viewed as a random set (\(L\), on the other hand, is determined).  
Of course, the set \( K \setminus L \) is not a uniform random set, but we will see below that it
is approximately uniform. Another virtue is that the progress of the Extension-Rotation procedure 
depends on $\G_{D_2}$ as opposed to $D_2$ and so it is identical for all digraphs in the same part of the partition,
up until the first time we check a vertex for membership in \( K \setminus L \).  This check 
leads to a refinement of the 
partition. 

We define a filtration 
\( \fee = \fee_0, \fee_1, \dots \) of \( \Omega_k \times [n]^k\).  
$\fee_0$ is the trivial partition with one part that is equal to \( \Omega_k \times [n]^k\).
The partition \( \fee_1 \) is determined by \( \Gamma_{D_1} \) and the third out-arcs of the vertices in \( L \): 
\( (D_1,A) \) and \( (D^\prime_1, A^\prime) \) are in the same part of \( \fee_1 \) if
\( \G_{D_1} = \G_{D^\prime_1} \) and the arcs in \(A\) with tails in \( L(D_1) = L(D_1^\prime) \)
are identical to the arcs in \( A^\prime \) with tails in \( L(D^\prime_1) = L(D_1) \).
The remaining partitions are then determined by the 
execution of the Extension-Rotation Procedure.
We refine the partition each time we consider an element of \( \eee \).
Thus, each time we arrive at a 2-factor \(F\) such that 
\( {\mathcal P}(F) \neq \emptyset \) and \( N(x) \subseteq X \) for all
\( x \in \eee \), where \(X\) is the vertex set of a longest path in \(F\)
and \( \eee \) is the set of endpoints of paths spanning \(X\), we will likely 
refine the partition many times.  This is because we will likely have to check many elements of
\( \eee \) before we find \( x \in \eee \) such that \( x \in K \setminus L \) and 
\( y_x \in \eee(x) \).  Let \( C_i \) be the first \(i\) vertices checked
in the Extension-Rotation Procedure.  Further, define 
\[ I_i = C_i \cap (K \setminus L)  \ \ \ \ \text{ and } \ \ \ \ O_i = L \cup ( C_i \setminus K ). \]
In words, \( I_i \) are the vertices among the first \(i\) checked that are in \(K \setminus L\)
and \( O_i \) is \(L\) together with the vertices among the first \(i\) checked 
that are out of \(K \setminus L\). If
ordered pairs \( (D_1, A) \) and \( ( D^\prime_1, A^\prime) \) are in the same part of
\( \fee_i \) then \( \G_{D_1} = \G_{D^\prime_1} \), \( I_i((D_1,A)) = I_i((D^\prime_1,A^\prime)) \), 
\( O_i((D_1,A)) = O_i((D_1^\prime,A^\prime)) \) and the arcs in \( A \) with tails in \( I_i \)
are identical to arcs in \( A^\prime \) with tails in \( I_i \).

If \( \cX \) is a part in \( \cF_i\)
then 
$\cD_1(\cX)=\set{D_1 \in \Omega_k:\;\exists A\in [n]^k\ such\ that\ (D_1,A) \in \cX}$.
Note that if $(D_1,A)$ is chosen uniformly at random from $\cX$ then $D_1$ will be chosen uniformly
at random from $\cD_1(\cX)$.

We show
that throughout the algorithm 
nearly every part \( \cX \) of the partition \( \fee_i \) has the property that
nearly 
all the digraphs in \( \cX \) satisfy certain structural
conditions (i.e.. small maximum degree, no small dense subgraphs, etc.).  
For \( \cX \) a part of
the partition \( \fee_i \) we define \( \G_\cX \) to be the underlying graph determined
by \( \cX \) (i.e. \( \G_{D_1 + A_2} \) together with the 
edges from \( E_3 \) that have been added by the extension-rotation procedure) and
\( A_{\cX} \) to be the out-arcs from \(A_3\) determined by \( \cX\).  Furthermore, define
\( C_{\cX} \), \( I_{\cX} \) and \( O_{\cX} \) to be the sets of vertices checked for membership
in \(K \setminus L\), determined to be in \(K \setminus L\), and determined not
to be in \(K \setminus L\), respectively, en 
route to the part \( \cX\).
We discard \( \cX \) if \( \G_\cX\) has does not have the desired 
structural properties or if 
too many digraphs \( D \in \cD_1(\cX)\) do not have the desired structural properties.
In Sections~\ref{sec:extrot}~and~\ref{sec:uni}, respectively, we formally define
two events: \( \cE_G \) a collection of graphs and \( \cE_D\)
a collection of digraphs, see \eqref{EG} and \eqref{ED} respectively.  We will prove that if \( (D_1, A) \) is 
chosen uniformly
at random from \( \Omega_k \times [n]^k \) then
\begin{gather}
\Pr( \exists B \subseteq A_3 \text{ such that }  \G_{D_1 + A_2 + B} \in \cE_G) = o(1)
\label{eq:graph} \\ 
\Pr( D_1 \in \cE_D \ \wedge \ \G_{D_1} \not\in \cE_G ) 
\le e^{ - \log^2 n }.  \label{eq:digraph} 
\end{gather}
At stage \(i\) of the algorithm we discard each part \( \cX \) of \( \fee_i \) such 
that \( \G_\cX \in \cE_G \).  
We also
discard \( \cX \) if
\( |\cD_1(\cX) \cap \cE_D |  > n^2 e^{ - \log^2 n} |\cD_1(\cX)| \).  
We bound the probability that digraph \(D\) is discarded at stage \(i\)
by applying the following version of 
Markov's inequality.  If \( \cE \) is an event and \( \cG \) is a partition of
our probability space then
\begin{equation}
\label{eq:markov}
\Pr[ \Pr[ \cE| \cG] \ge \e ] = \Pr [ E [ 1_\cE | \cG ] \ge  \e] \le \frac{1}{\e} E [ E[ 1_\cE| \cG]] 
= \frac{ E[ 1_\cE]}{ \e} = \frac{ \Pr(\cE)}{ \e}.
\end{equation}
Applying (\ref{eq:markov}) with \( \e = n^2 e^{ -\log^2n} \) we see that the probability that
there exists an index \(i\) such that our randomly chosen \((D_1,A)\) is discarded at stage \(i\)
due to \eqref{eq:digraph}
is bounded above by \( 1/n \).

So, we restrict our attention to parts \( \cX \) of the partition \( \fee_i \)
in the filtration given by the Extension-Rotation algorithm such that 
\( \G_\cX \not\in \cE_G \) and the probability \( D_1 \) is in \( \cE_D \)
where \(D_1\) is chosen uniformly at random from \(  \cD_1(\cX)\)
is at most \( n^2 e^{-\log^2n} \).  The remainder of our analysis hinges on
the following two Lemmas.  Of course, we must only show that with high probability
the Extension Rotation Procedure does not arrive in a situation were the longest path in the current
2-factor cannot be completed to a cycle using the arcs in \(A_3\).  To this end,
we first show that there are many end-points of paths that span the vertex set
spanned by the longest path in the 2-matching.
\begin{lemma}
Let \(G\) be a fixed graph on \(n\) vertices such that \(\delta(G) \ge 3 \).  
Suppose further 
that \(P\) is a maximal path in \(G\), \(x\) is an
endpoint of \(P\), \(X\) is the set of
vertices spanned by \(P\) and \( \eee(x) \) is the set of end-points (other than \(x\)) 
of paths that span \(X\) and have \(x\) as one end-point.  Then we have
\label{lem:extrot}
\[ G \not\in \cE_G  \ \ \ \ \Rightarrow \ \ \ \ \eee(x) >  \frac{ n }{ 100 }. \]
\end{lemma}
\noindent
Lemma~\ref{lem:extrot} is proved in Section~\ref{sec:extrot}.
One consequence of Lemma~\ref{lem:extrot} is that each time we find \( x \in \eee \) such that
\( x \in K \setminus L \) then the probability that the third edge out of \(x\) completes the cycle
is bounded away from zero.  Thus, \whp\ the number of cycles we close among the
first \(j\) elements of \(K \setminus L\) found is at least \( j/ 200 \) for, say, \( j \ge \sqrt{n} \).  
Therefore, the number of elements of \(K \setminus L\) found before the successful termination of
the Extension-Rotation Procedure is \whp\ at most \( 1200 n / \log n \).
It remains to 
show that each \(x \in \eee \) we consider (that is in neither \( I_\cX\) nor \( O_{\cX} \))
has a good chance of being in \(K \setminus L \).
\begin{lemma}
\label{lem:uni}
Suppose \( \cX\) is a part of the partition \( \fee_i \) such that \( \G_\cX \not\in \cE_G\) 
and \[ \left| \cE_D \cap \cD_1(\cX) \right|  < n^2 e^{- \log^2 n } \left| \cD_1(\cX) \right|. \]  
Suppose further that \( |I_\cX| 
\le  \frac{1200 n}{ \log n} \)
and \( |O_\cX| = o( n ) \).  If \(D\) is chosen uniformly at random from \( \cD_1(\cX)\)
then
\[\left| \left\{ x \in V : \Pr( x \in K(D) ) \le \frac{ k}{n} \cdot \frac{1}{ 4000 (\log \log n)^3 }
\right\} \right| = o(n). \]
\end{lemma}
\noindent
Lemma~\ref{lem:uni} is proved in Section~\ref{sec:uni}.
It follows from this Lemma that \whp\ the number of elements of \(K \setminus L\) found among the first \(i\) 
vertices queried is at least \( i/ ( \log^{1/2} n \cdot 8000 (\log \log n)^3 ) \) for, say, \(i \ge \sqrt{n}\).
Indeed, as we explore $\eee$ we can ignore the $o(n)$ vertices in $I_{\cX}\cup O_{\cX}$ and the probability a 
queried vertex is in \( K \setminus L \) is at 
least
\[ \frac{2}{3} \cdot \frac{1}{ 4000 \log^{1/2} n ( \log \log n)^3 } \]
as the probability that the queried vertex is among the `bad' vertices described in Lemma~\ref{lem:uni} is
at most 1/3.  Thus, the number of elements of \( K \setminus L \) identifies among the first \(i\) vertices 
queried 
%
dominates $Bin(i,1/(6000 (\log \log n)^3 \log^{1/2}n ))$.
Let \( N \) be the number of vertices queried for membership in \(K \setminus L\) 
before completion of
the Extension-Rotation Procedure.  Assuming the rare events discussed above do not occur, we have
\[ \frac{N} { \log^{1/2} n \cdot 8000 (\log \log n)^3 } \le \frac{1200 n}{ \log n}. \]
and so
\[ N = O \left( n \cdot \frac{ ( \log \log n)^3}{ \log^{1/2} n} \right). \]
Therefore, \whp\ we have \( |O_\cX| = o(n) \) throughout the execution of 
Extension-Rotation (inductively justifying our application of Lemma~\ref{lem:uni}).


\section{A simple 2-matching 
with few components}\label{sec:2matching}

We apply the following consequence of the Tutte-Berge matching formula 
(see Theorem~30.7 in Shrijver \cite{Sch}).  If \(G = (V,E) \) is a graph and \( A,B \) are disjoint subsets
of \(V\) then we let \(e(A,B) \) denote the number of edges in \(E\) that intersect both \(A\) and \(B\)
and \( e(A) \) denote the number of edges within \(A\).  The {\bf size} of a simple 2-matching 
\(F\) is the number of 
edges in \(F\).
\begin{theorem}
\label{thm:tb}
Let \(G = (V,E) \) be a graph.  The maximum size of a simple 2-matching in \(G\) is equal 
to the minimum value of 
\begin{equation}
\label{eq:tutte}
|V| + |U| - |S| + \sum_X \left\lfloor \frac{ e(X,S)}{2} \right\rfloor, 
\end{equation}
where \(U\) and \(S\) are disjoint subsets of \(V\), with \(S\) an independent set, and where \(X\)
ranges over the components of \(G - U -S \).
\end{theorem}
\noindent
  We make extensive use of the following observation:
\begin{equation}
\label{eq:easy}
S\subseteq \binom{[n]}{2} \ \ \ \ \ \Rightarrow \ \ \ \ \ 
\Pr(S\subseteq E(G_{3-out}))\leq \bfrac{6}{n}^{|S|}.
\end{equation}
This is easily seen to be true for $|S|=1$ and after this, knowing
that $S'\subseteq E(G_{3-out})$ only reduces the probability that
$e\notin S'$ is also in $E(G_{3-out})$.

We define \( \cE_{\rm cycles} \)
to be the collection of graphs on vertex set \([n] \) for which there are at least \( n^{3/4} \) vertex 
sets \( C \) such that \( |C| \le \frac{ \log n}{3.9} \), \( G[C] \) is connected and \( G[C] \) has at least
\( |C| \) edges.
\begin{lemma}
\label{lem:cycles}
With high probability \( G_2 \not\in \cE_{\rm cycles}\).
\end{lemma}
\begin{proof}
We compute an upper bound on the expected number of small complex components in \( \G_D \) 
where \(D\) is drawn uniformly at random from 
\( \Omega \):
\[ \sum_{c=3}^{ \frac{\log n}{3.9} } \binom{n}{c} c^{c-2} \binom{c}{2} \bfrac{ 6}{n}^c \le 
\sum_{c=3}^{ \frac{\log n}{3.9}} ( 6e)^c = O \left( (6e)^{ \frac{ \log n}{3.9}} \right). \]
The Lemma then follows from Markov's inequality.
\end{proof}

\begin{lemma}
\label{lem:ind}
With high probability \( \alpha( G_2) \le 0.415 n \).
\end{lemma}
\begin{proof}
Let \( Z \) be the size of the largest independent set in \( G_1 = \G_{D_1} \).  Applying the union
bound (and noting that a maximum independent set is also a dominating set) we have
\begin{equation*}
\begin{split}
\Pr( Z \ge \ell ) & \le \sum_{ i \ge \ell} \binom{n}{i} \left( 1 - \frac{i}{n} \right)^{3i 
- k}
\left[ 1 - \left( 1 -\frac{1}{n-i} \right)^{3i} \left( 1 - \frac{i}{n} \right)^3 \right]^{n-i 
- k }
\end{split}
\end{equation*}
We checked numerically that the function
\[ f(x) = \frac{ ( 1 -x)^{3x} \left( 1 - (1-x)^3 e^{-3x/(1-x)} \right)^{1-x} }{ x^x (1-x)^{1-x} } \]
is strictly less than one for \( x \ge 0.415 \). We conclude that \( Z \le 0.415 n \) with high
probability.  As \( G_1 \) is a subgraph of \( G_2\), we have the desired result.
\end{proof}

Rather than work with the event implicit in Theorem~\ref{thm:tb} 
directly, we consider a 
weaker event.
Let \( \cG \) be the event that there exist disjoint vertex sets \( U,R \) such that
\begin{itemize}
\item[(i)] \( V \setminus (U \cup R) \) can be partitioned into sets \( S,T \) such that
\( G[T] \) is a forest and \( G[S] \) is an independent set.
\item[(ii)] \( |S| \ge \max\{ |U|, 2n/( \log n ) \} \),
\item[(iii)] \( e( V \setminus( U \cup R)) + e( R, V \setminus( U \cup R)) 
= e(S \cup T) + e( R, S \cup T) \le |V| + |S| - 3|U| - |R| \).
\end{itemize}
We now argue that \( G \not\in \cG \cup \cE_{\rm cycles} \) implies that \(G\) has a 
simple 2-matching with at most
\( 6n/ \log n \) components.  Suppose \( G \not\in  \cG \cup \cE_{\rm cycles} \). 
Let \(F\) be a simple 2-matching of maximum size.  Since 
\(G \not\in \cE_{\rm cycles} \), \(F\) has at most
\( 4n/ \log n\)  cycle components, and therefore the number of components in \(F\)
is at most \( 4n/ \log n\) plus \(n\) minus the size of \(F\).  So, it suffices
to show that \(G\) has a simple 2-matching with at least \( n - 2n/ \log n \) edges.
Let \(U\), \(S\) be an arbitrary pair of disjoint subsets of \(V\) where
\(S\) is an independent set.  
We may assume \( |S| \ge \max\set{2n/ \log n, |U|} \) (otherwise the expression in (\ref{eq:tutte}) 
is clearly at least $n-2n/\log n$).
For the given \(U\) and \(S\), let \(R\) be the set of vertices in non-tree 
components in \( G_2[ V \setminus (U \cup S) ] \).  Set \(T = V \setminus (U \cup S \cup R) \).  Note that
the graph induced on \(T\) is a forest.  Suppose this is a forest with \(p\) components.  Now, since 
\( G \not\in \cE_{\rm cycles} \) implies
that there are at most \( 4n/ \log n \) components in \( G[R] \), we have
\[ \sum_{X} \left\lfloor  \frac{e(X,S)}{2} \right\rfloor  \ge \frac{ e( S,T) + e(S,R) - p 
-  \frac{4n}{\log n}}{2}. \]
On the other hand, \( G \not\in \cG\) implies that condition (iii) does not hold and we have
\begin{equation*}
\begin{split}
e(S,T) + e(S,R) & > |V| + |S| - 3|U| -|R| - e(T) \\
& = |V| + |S| - 3|U| -|R| - |T| + p  \\
& = 2|S| - 2|U| + p .
\end{split}
\end{equation*}
Therefore,
\[ \sum_{X} \left\lfloor  \frac{e(X,S)}{2} \right\rfloor >  |S| - |U| - \frac{2n}{ \log n}.  \]
Applying Theorem~\ref{thm:tb}, we see that there exists a simple 2-matching \(F\) in \(G\) with at 
at least \(n - 2n/ \log n\) edges.

It remains to show that \( G_2 \not\in \cG \) with high probability.
We apply the first moment method.
Throughout these computations we take \(K\) to be a fixed set (the set \(L\), however, 
is not determined).  For ease of notation we set \( u = |U|, r= |R|, s =|S|,\) etc.
We consider 3 cases based on the value of \( u/n \).

{\bf Case 1.} \( u \le \frac{n}{e^{18}} \).

\noindent
Here we take special care with the vertices in \( S \cup T \) that are in \(K\).
For disjoint sets \(U ,Z ,W \) such that \(U, Z \subseteq V \) and 
\( W \subseteq K \) let \( \cG_{U,Z,W} \) be 
the event that conditions (i), (ii) and (iii) are satisfied with 
\[ S \cup T = Z \cup W \ \ \ \ \ \ \  ( S \cup T) \cap L = Z \cap K \ \ \ \ \ \ \
(S \cup T) \cap (K \setminus L) = W. \]
In words, \( Z \) is the set of vertices in \( S \cup T \) of out-degree \(3\) 
in \( D_1 + A_2 \) and \(W\) is the set of vertices in \(S \cup T \) of out-degree 2 in \( D_1 + A_2 \).
Note that this event includes conditions on arcs that point at the 
vertices in \( Z \cap K \) and \(W\); 
in particular, there is at least one arc directed at each vertex in \( W \).
Let \( \alpha \) be the number of arcs directed from \( Z \cup W = S \cup T \)
to \( S \cup T \cup R \) and let
\( \beta \) be the number of edges that are directed 
from \(Z \cup W = S \cup T \) to \(U\).  Finally, let \( \gamma \) be the number of arcs directed 
from \(R= V \setminus (U \cup S \cup T)\) into vertices in \(W\).  

Before proceeding with the calculation, we establish an upper bound on \( \alpha \) 
(and hence a lower bound on \( \beta \)).  Note that, as we assume 
condition (iii) is satisfied, we have
\begin{multline*}
\alpha + \gamma \le  e( S \cup T ) + e( R, S \cup T ) \le |V| + |S| - 3|U| - |R| 
= 2s -2u +t. 
\end{multline*}
Therefore
 \[ \alpha \le 2s-2u+t - \gamma  \ \ \ \text{ and } \ \ \ 
\beta \ge s + 2t - w + 2u + \gamma . \]
The latter follows from $\a+\b= 3s+3t-w$.

For fixed \( U,Z,W \) we have
\begin{eqnarray}
 \Pr( \cG_{U,Z,W}) & \le & \sum_{\beta, \gamma}  \binom{ 3(s+t) - w}{\beta} \bfrac{u}{n}^\beta 
\bfrac{ 3u + \alpha }{n}^{w - \gamma}  \label{ins} \\
& \le & 9n^2 \cdot 2^{ 3(s+t) -w} \bfrac{ u}{n}^{s+2t + 2u} \bfrac{ u + 2s+t}{ u}^{w - \gamma}. \nonumber
\end{eqnarray}
Note that the last term in (\ref{ins}) accounts for the condition that there must be an arc directed into 
each vertex in \(W\):  There are at least \( |W| - \gamma \) vertices in \(W\) that are not `covered' by vertices
in \(R\), and each of these vertices must be covered by one of the arcs out of \(U\) or one of the \( \alpha \) arcs
out of \( S \cup T \) that point to vertices in \( S \cup T \cup R \).

We bound the probability that there exists a pair \(U,R\) satisfying the conditions of the event \( \cG \) with 
\( |U| \le n/e^{18} \) by summing over all possible choices of \( U,Z \) and \(W\):
\begin{equation}
\begin{split}
Pr\left( \bigcup_{ U,Z,W } \cG_{U,Z,W} \right) 
& \le \sum_{u,z,w } \binom{n}{u} \binom{n}{z} \binom{k}{w}
\\
& \hskip1.5cm \cdot
 9n^2 2^{ 3(s+t) -w} \bfrac{ u}{n}^{s+2t + 2u} \bfrac{ u + 2s+t}{ u}^{w - \gamma} \\
& \le 9n^2 \sum_{u,z,w} \bfrac{ne}{u}^u \bfrac{ ne}{ s+t-w}^{s+t-w} \bfrac{ ke}{w}^w \\
& \hskip2cm \cdot
8^{s+t} \bfrac{1}{2}^w \bfrac{u}{n}^{s+2t+2u}  \bfrac{ 3(s+t)}{ u}^{w} \\
& \le  9n^2 \sum_{u,z,w}  \bfrac{3k}{2n}^w \bfrac{ eu}{n}^u \bfrac{ 16e u}{s+t}^{s+t} \bfrac{s+t}{u}^w.
\label{eq:smallu}
\end{split}
\end{equation}
Note that to get the final summand we dropped a factor $(u/n)^t$ and we used the bound
\[ \bfrac{ s+t}{ s+t-w}^{s+t-w} \bfrac{ s+t}{w}^w \le 2^{s+t}. \]
(This can be viewed as a special case of equation (\ref{ineq}), which is
used extensively in Section~\ref{sec:extrot}.)
We bound the final summand in (\ref{eq:smallu}) by considering two subcases.  Recall that we assume 
\( s> 2n/ \log n \).  If \( s+t \ge 32eu \) then we write
\[ \bfrac{ 16e u}{s+t}^{s+t} \bfrac{s+t}{u}^w  \le  \bfrac{ 1}{2}^{s+t-w} (16e)^w . \]
If \(w\) is a positive proportion of \( s+t\) then the \( \frac{3k}{2n} \) term suffices to 
establish the desired upper bound, and otherwise the \( (1/2)^{s+t-w} \) term suffices.  If, on the other
hand, \( s+t< 32 e u \) then $\bfrac{3k}{2n}^w\bfrac{s+t}{u}^w\leq 1$ and we write
\[ \bfrac{ eu}{n}^u \bfrac{ 16e u}{s+t}^{s+t} \le \left( e^{17} \frac{u}{n} \right)^u, \]
and again we have the desired upper bound.

\noindent
{\bf Case 2.} \( n/ e^{18} \le u \le 0.015n \).

Suppose first that \( s+t = n-u-r \ge 8eu \).  
As above, we let \( \beta \) be the number of arcs directed 
from \( S \cup T \)
to \( U \). We have \( \beta \ge s + 2t + 2u - o(n) \).
The probability that there exist \(U,R\) that satisfy condition (iii) with
\( n - |U| - |R| \ge 8e|U|\) and \( n/e^{18} \le |U| \le 0.015 n \) is 
bounded above by (we write $n-u-r=s+t$) 
\begin{equation*}
\begin{split}
\sum_{u,r,\beta} \binom{n}{u} \binom{n}{s+t} \binom{3(s+t)}{ \beta} \bfrac{u}{n}^\beta 
& \le e^{o(n)} \sum_{u,r} \bfrac{ ne}{u}^u \bfrac{ ne}{s+t}^{s+t} 2^{3(s+t)}  \bfrac{u}{n}^{s + 2t + 2u} \\
& \le e^{o(n)} \sum_{u,r} \bfrac{ u}{n}^t \bfrac{ eu}{n}^{u} \bfrac{ 8eu}{s+t}^{s+t}.
\end{split}
\end{equation*}
Note that the \(  \bfrac{ eu}{n}^{u} \)
term is exponentially small and dominates the $e^{o(n)}$ term and the other terms are at most 1.

Suppose now that \( s +t = n-u-r < 8eu \).  Here we add the 
consideration of the edges from \(R\) to \(S \cup T\) to the previous bound.  Let \( \gamma \) 
be the number of arcs directed from \(R \) to \( S \cup T \). Now we have 
\( \beta  \ge s + 2t + 2u + \gamma - o(n) \).  The probability that there exist \(U,R\) 
that satisfy condition (iii) with
\( n - |U| - |R| < 8e|U|\) and \( n/e^{18} \le |U| \le 0.015 n \) is 
bounded above by 
\begin{equation}
\begin{split}
\label{eq:foralan}
\sum_{u,r,\gamma,\beta} & \binom{n}{u} \binom{n}{s+t} \binom{3(s+t)}{ \beta} \bfrac{u}{n}^\beta 
\binom{3r}{ \gamma} \bfrac{s+t}{n}^\gamma \left( 1 - \frac{s+t}{n} \right)^{3r - \gamma}\\ 
& \le e^{o(n)} \sum_{u,r,\gamma} \bfrac{ ne}{u}^u \bfrac{ ne}{s+t}^{s+t} 2^{3(s+t)}  
\bfrac{u}{n}^{s + 2t + 2u + \gamma} \\
& \hskip2cm \times
\bfrac{3re}{ \gamma}^\gamma \bfrac{s+t}{n}^\gamma \left( 1 - \frac{s+t}{n} \right)^{3r - \gamma}.
\end{split}
\end{equation}
We write
$$
\left( \frac{  3re \cdot \frac{ s+t}{n} \cdot \frac{u}{n}}{ \gamma} \right)^\gamma 
\left( 1 - \frac{s+t}{n}  \right)^{3r-\gamma} 
\le e^{3r \frac{ s+t}{n} \frac{u}{n} }
 e^{ - \frac{s+t}{n}( 3r - \gamma)}  \le e^{ - \frac{ s+t}{n} ( 3n - 6u - 3t - 3s - \gamma )}.  
$$
Since \( 0 \le \alpha \le 2s-2u + t - \gamma \), we have \( \gamma \le 2s - 2u + t  \) and
\[ 6u + 3t + 3s + \gamma \le 4u + 4t + 5s \le (4 + 5 \cdot 8e ) u \le 1.99 n. \]
So, we can bound the expression in (\ref{eq:foralan}) by
\begin{equation*}
\begin{split}
e^{o(n)}\sum_{ u,r} \bfrac{ ne}{u}^u \bfrac{ ne}{s+t}^{s+t} \bfrac{8}{e}^{s+t} 
\bfrac{u}{n}^{s + 2t + 2u} &
\le e^{o(n)}\sum_{u,r} \bfrac{eu}{n}^u  \bfrac{u}{n}^{t} \bfrac{ 8u  }{s+t}^{s+t} \\
& \le e^{o(n)}\sum_{u,r} \bfrac{eu}{n}^u   e^{  u \bfrac{8}{e } } 
\end{split}
\end{equation*}
As \( 0.015 \cdot e \cdot e^{8/e} < 1 \), we have the desired bound.

\noindent
{\bf Case 3.} \( 0.015 n \le u \).

Let  \(U,R\) (and \( S \cup T = V \setminus( U \cup R) \)) be fixed disjoint sets such that 
\( |U| \ge 0.015 n \).  Let \( \alpha \) be the number of
arcs directed from \( S \cup T \) to \( S \cup T \cup R \) and \( \z \) 
be the number of arcs directed from \( R \) to \(S \cup T \).  
The probability that there exist such sets \(U,R\) with the specified number of arcs is bounded above
by
\begin{multline}
\label{eq:another4alan}
\sum_{u,r,\alpha, \z} 
\binom{n}{u} \binom{n-u}{r} \binom{ 3(n-u-r)}{ \alpha}\left( 1 - \frac{u}{n} \right)^{ \alpha }
 \bfrac{u}{n}^{ 3(n-u-r) -\alpha} \cdot \\
 \binom{ 3r }{ \z } 
\left( 1 - \frac{u + r}{n}\right)^\z \bfrac{u+r}{n}^{3r - \z }.
\end{multline}
As we are bounding the probability 
of the event \( \cG \), we assume (in particular) that condition (iii) holds; that is, we assume,
\[ \alpha + \z = e(S \cup T) + e(S \cup T, R) \le n + s - 3u - r, \]
which implies
\begin{equation}\label{xx1}
 \alpha \le  n + s - 3u - r - \z = 2s - 2u +t - \z \le 2s+t-2u.
\end{equation}
Now the expression 
\begin{equation}\label{xx2}
 \binom{ 3(s+t)}{ \alpha} \left( 1 -\frac{u}{n} \right)^\alpha \bfrac{u}{n}^{-\alpha}
\end{equation}
is increasing for 
$\a\leq 3(1-u/n)(s+t)-u/n$. 
Given \eqref{xx1} and \( u \le s \le \alpha( G_2 ) \le 0.5 n \) (by Lemma~\ref{lem:ind}) we have
\begin{eqnarray*}
3(1-u/n)(s+t)-u/n-(2s+t-2u)&=&(1-3u/n)s+(2-3u/n)t+2u-u/n\\
&\geq&(1-3u/n)s+2u-u/n\\
&\geq&0.
\end{eqnarray*}
Thus in our range of interest, the expression \eqref{xx2} is increasing in $\a$.

Therefore, the probability that there exist disjoint sets 
\(U,R\) with \( |U| = u \) and \( |R| =r \) that satisfy conditions (i), (ii) and (iii) is bounded above by
\begin{multline}
\label{eq:final}
e^{o(n)} \sum_{u,r,\z} \binom{n}{u} \binom{n-u}{r} \binom{ 3(n-u-r)}{  n + s - 3u - r - \z }
\left( 1 - \frac{u}{n} \right)^{ n + s - 3u - r - \z }
 \bfrac{u}{n}^{2n -s -2r + \z - o(n)} \\\
 \binom{ 3r }{ \z } 
\left( 1 - \frac{u + r}{n}\right)^\z \bfrac{u+r}{n}^{3r - \z - o(n)},
\end{multline}
where we set \( s = \min \{ n -u-r, 0.415 n \} \).  Note that 
the only role \(s\) plays in (\ref{eq:another4alan}) is in the upper bound on \( \alpha \).  So, we 
take the maximum possible value of \( \alpha \) by taking the maximum value for \(s\) (and this has 
no impact on the rest of the expression).  Also note that we apply
Lemma~\ref{lem:ind} to set an absolute upper bound on the \(s\).  

We turn now to calculations. We re-name $u\to xn,r\to yn,\g\to
zn$. Then for $a\geq b$, we let $\Bin(a,b)=\frac{a^a}{b^b(a-b)^{a-b}}$ be a
replacement for a corresponding binomial.
We now define the function of three real variables
\begin{multline}\label{function}
g(x,y,z) = \Bin(1,x)\cdot \Bin(1-x,y)\cdot \Bin(3(1-x-y),\a)\cdot\Bin(3y,z)\cdot\\  
x^{ \beta(x,y,z)} \cdot (1-x)^{ \alpha(x,y,z)}
\cdot ( 1 - x -y)^z \cdot ( x+y)^{3y - z}
\end{multline}
where
\begin{gather*}
\alpha(x,y,z) = 1 + s(x,y,z) - 3x - y -z \ \ \ \ \ \ \ \ \
\beta(x,y,z) = 2 - s(x,y,z) - 2y + z  \\
s(x,y,z) = \min \{ 1 - x -y, 0.415 \}. 
\end{gather*}
Note that $y\leq 1-2x$ is equivalent to $u\leq s+t$.
\begin{claim}
\label{cl:calc}
If \( 0.015 \le x \le 0.415\), \( 0 \le y \le 1 - 2x \) and \( 0 \le z \le 3y \) then
\( g(x,y,z) \le 0.995 \).
\end{claim}
Note that Claim~\ref{cl:calc} implies that the expression in equation (\ref{eq:final}) 
is exponentially 
small for all relevant choices of \( r,u,\z \), and Lemma~\ref{lem:2match} follows.

It remains to prove Claim~\ref{cl:calc}.  Of course, this can be achieved by a 
direct assault
using calculus (e.g. concavity is useful).  However, the authors feel that a simple check
by computer is more compelling.  Posted at {\tt www.math.cmu.edu/$\sim$tbohman} the reader will
find a short
program in C that divides the 3-dimensional region into small cubes (side length 0.00125)
and bounds the function over each cube, thereby bounding the function over the entire region.
The numerical bounds on the expression in \eqref{function} are achieved 
by bounding each term in \eqref{function} individually over each sub-cube in a 
trivial manner.  Problems with numerical rounding are avoided



\section{Rotations: Proof of Lemma~\ref{lem:extrot}}
\label{sec:extrot}

We begin by defining the event \( \cE_G \).  This event is 
comprised of three properties: large maximum degree, small dense sets 
and a more complicated non-monotone condition for larger sets. Set
\[ \Delta_0 = \frac{ 2 \log n}{ \log \log n}. \]
The event 
\begin{equation}\label{EG} 
\cE_G =\set{\D(G) \ge \D_0}\cup \cE_{\rm small} \cup  \cE_{\rm pieces}
\end{equation}
where \( \cE_{\rm small} \) is
the event that there is a vertex set \(S\) and an integer \(i\) such that 
\( 1 \le i \le 10 \), \( |S| < \frac{i}{4} \log n \) and the graph induced by \(S\) is
connected and spans at least \( |S| + i \) edges. $\cE_{\rm pieces}$ is defined immediately following the 
proof of Lemma \ref{lem:g1} below.

The events \( \D(G) \ge \D_0 \) and
\( \cE_{\rm small} \) are increasing and so we can bound the probability that there exists \( B \subseteq A_3 \) such that 
\( \G_{G_1 + A_2 + B} \) is in one of these events by working with 
\( \Omega \) itself.  
\begin{lemma}
\label{lem:g1}
If \( D \) is chosen uniformly at random from \( \Omega \) then
\[  \Pr( \D(\G_D) \ge \D_0 \ \vee \ \G_D \in \cE_{small} 
) = o(n^{-1/2}).  \]
\end{lemma}
\begin{proof}
The degree of a vertex in \( \Gamma_D \) is distributed as
\( 3 + \Bin(3n, 1/n) \).  Therefore
\[ \Pr( \Delta( \Gamma_D) \ge \Delta_0) \le n \binom{3n}{ \D_0 -3} \frac{1}{n^{\Delta_0-3}}
= o(n^{-1/2}). \]
Another application of the union bound gives
\begin{equation*}
\begin{split}
\Pr( \G_D \in \cE_{\rm small} ) & \le 
\sum_{s = 4}^{ \frac{10}{4} \log n} \binom{n}{s} s^{s-2} 
\binom{s}{2}^{ 1 + \left\lceil \frac{s}{ \frac{1}{4} \log n} \right\rceil}
\left( \frac{6}{n} \right)^{s+ \left\lceil \frac{s}{ \frac{1}{4} \log n} \right\rceil } \\
& \le \sum_{s=4}^{ \frac{10}{4} \log n}
\left( \frac{ ne }{s} \cdot s \cdot \frac{6}{n} \right)^s 
\cdot \left(  \frac{3 s^2}{n} \right)^{ \left\lceil \frac{s}{ \frac{1}{4} \log n} \right\rceil}\\
& = o( n^{-1/2} ).
\end{split}
\end{equation*}
\end{proof}

Let
$N(S)=\{y\notin S:\;\exists x\in S\mbox{ such that }
y\mbox{ is adjacent to }x\}$ 
be the set of neighbors of a vertex set $S$.
Define \( \cE_{\rm pieces} \) to be the event that there exist sets \(S_1,S_2,T_1,T_2\)
such that 
\begin{itemize}
\item[(i)] \(S_1 \) and \(S_2\) are disjoint with \( n^{1/30} \le |S_1| + |S_2| \le \frac{ n}{100} \).
\item[(ii)] For each \( x \in S_1 \) there exists \( y_x \in S_1 \cup T_1 \) such that
\( \{x,y_x\} \) is an edge, with the additional property that if \( u,x \in S_1 \) 
and \( u = y_x \) then \( x \neq y_u \).
\item[(iii)] For each \( x\in S_2 \) there exist distinct \(y_x,z_x \in T_2\) such that
\( \{x,y_x\} , \{ x,z_x\} \) are edges.
\item[(iv)] \( T_1 = \{ y_x : x \in S_1 \text{ and } y_x \not\in S_1 \} \) and 
\( T_2 = \{ y_x, z_x : x \in S_2 \} \).
\item[(v)] \( N(S_1),N(S_2) \subseteq T_1 \cup T_2 \) and 
\(S_2\) is an independent set.
\end{itemize}
\begin{lemma}
\label{lem:pieces}
Let \( ( D_1, A ) \) be chosen uniformly at random from \( \Omega_k \times [n]^k \).
\[ Pr\left( \exists B \subseteq A_3 : \G_{D_1 + A_2 + B} \in \cE_{\rm pieces} 
\right) \le e^{-\Omega( n^{1/30})}. \]
\end{lemma}
The proof of Lemma~\ref{lem:pieces} is given in Section~\ref{sec:pieces} below.  Note that 
Lemmas~\ref{lem:g1}~and~\ref{lem:pieces} imply (\ref{eq:graph}).  We now proceed
with the proof of Lemma~\ref{lem:extrot}.

Let  $P$ be $ x=x_0,x_1, x_2, \dots$ and \( S = \eee(x)\).
We show \(|S|\) is large in two stages.

\noindent
{\bf Stage 1:} \( |S| \ge n^{1/30} \).

Let $T_i$ be the
set of end-points (other than \(x_0\)) 
of paths that can be obtained from $P$ by
$i$ rotations with $x_0$ fixed.  We view the \( T_i\)'s as the level sets in
a tree rooted at \( x_0 \).  Note that the children of a vertex 
\( y \in T_i \) in this
tree are second-neighbors of \(y\) in \(G \).
Since each end-point \( y \in T_i \) yields \( d(y)-1 \) 
end-points in \( T_{i+1} \) and \( \d( G ) \ge 3 \)
we have
\begin{equation}\label{TT}
|T_{i+1}| \geq  2|T_i | - \d_i,
\end{equation}
where $\d_i$ accounts for vertices 
in \( T_i \) that have second-neighbors in common.

If for some $\ell,\r$ we have \( \sum_{i=1}^\ell \d_i = \r  \) then there exists a connected subgraph \(H\) with \( v \le 4 \ell \r \) 
vertices and at least \( v +\r  -1 \) edges.  Since
\( G \not\in \cE_{\rm small} \), it follows that \( \r \le 10 \) implies
\[ 4\ell \r \ge v \ge \frac{\r-1}{4} \cdot \log n \ \ \ \Rightarrow \ \ \
\ell \ge \frac{\r-1}{\r} \cdot \frac{1}{16} \cdot \log n. \]
Therefore,
\[ \sum_{i=1}^{ (\log n)/18} \d_i \le 9 \] 
and 
\[ |S| \ge \left| T_{ (\log n)/18 } \right| \ge 2^{ \frac{ \log n}{18} - 9} \ge n^{1/30} \]
for \(n\) sufficiently large.

\noindent
{\bf Stage 2:} \( |S| \ge  \frac{ n }{100} \).

Suppose that the vertices in $S$ induce a collection of sub-paths
$P_1,P_2,\ldots,P_m$ of $P$
and that $P_i$ contains $s_i$ vertices.  Let \( Y \) be the 
collection of paths \(P_i \) that consist of a single vertex. (So, if \( s_i =1 \) 
then \( P_i \in Y \).  Note that 
for these paths we abuse notation by referring 
to a single vertex as a path.)
P\'osa \cite{Po} proves that
\begin{equation}\label{Posa}
x_i\in N(S) \ \ \ \Rightarrow \ \ \ x_{i-1} \in S \text{ or } x_{i+1}\in S.
\end{equation}
Note that (\ref{Posa}) implies
\begin{equation}
\label{Posa2}
u\in Y \ \text{ and } \ v \in Y \ \ \ \Rightarrow \ \ \  \{u,v\} \not\in E(G).
\end{equation}
Set 
\begin{align*}
S_2 &= Y, & T_2 &= \{ x_i : x_{i-1} \in S_2 \text{ or } x_{i+1} \in S_2 \} \\
S_1 &= S \setminus Y, & T_1 &= \{ x _i : x_i \not\in S \text { and } 
(  x_{i-1} \in S_1 \text{ or } x_{i+1} \in S_1 ) \}.
\end{align*}
In words, \( S_2 \) is the set of vertices that are singleton 
paths in the collection \( P_1, \dots, P_m \) and \(S_1 \) is the set of vertices in longer paths.
The sets \(T_2 \) and \( T_1 \) are the vertices adjacent along \(P\) 
to the end-vertices of the paths in \( P_1, \dots, P_m \).  \( S_1, S_2, T_1 \) and \( T_2 \)
satisfy conditions (ii), (iii), (iv) and (v) in the definition of the event \( \cE_{\rm pieces} \) 
and \( G \not\in \cE_{\rm pieces} \). Here for $x\in S_2$, $y_x,z_x$ are its two neighbors on $P$.
For each non-singleton path \( P_i \) the vertex \(x \in P_i\) closest to
\(x_0\) along \(P\) lets \( y_x\) be its neighbor along \(P\) closer to \( x_0\) and all other  
\(v \in P_i\) let $y_v$ be the neighbor along $P$ in the direction away from $x_0$. 
It follows that \( |S| = |S_1| + |S_2| \ge \frac{ n }{100} \).

\subsection{Proof of Lemma~\ref{lem:pieces}}
\label{sec:pieces}

We apply the first moment method.

Let \(s = |S_1| + |S_2|\) and \( \n = |S_2| \).  

We view \(K\) as a fixed set.  The set \(L\), on the other hand, will be random.
We then
let \( K^\prime \subseteq K \) be the tails of the arcs in \( A \setminus ( A_2 \cup B ) \).    
So \( K^\prime \) is the set of vertices (for a given choice of \(B\)) for which the third out-arc
is still not in the graph.  For \( X \subseteq K \) and sets \( S_1, S_2, T_1, T_2 \subseteq [n] \)
let \( \cE_{X, S_1, S_2, T_1, T_2 } \) be the event that \(K^\prime = X \) and 
\( S_1, S_2, T_1, T_2 \) satisfy conditions (i)--(v).  We bound the probability of the union of 
all events of this form using the first moment method.
A key property of this event is that, since we set \( X = K^\prime \), each vertex in
\(X\) has in-degree at least 1.

We do not work with the events \( \cE_{X, S_1, S_2, T_1, T_2 } \) directly.  Rather we consider the 
events \( \cF_{Y, S_1, S_2, T_1, T_2 } \) where \( Y \subseteq (S_1 \cup S_2) \cap K \) defined by
\[  \cF_{Y, S_1, S_2, T_1, T_2 } 
= \bigvee_{  \substack{X \subseteq K : \\ 
X \cap (S_1 \cup S_2) = Y}  }  \cE_{X, S_1, S_2, T_1, T_2 }. \]
In words, we consider the event where we specify membership in \(K^\prime \) only for those elements of
\(K\) that are in \( S_1 \cup S_2 \).  We choose \( Y, S_1, S_2 \) using the following method 
(in fact, we
break \(S_2\) into more parts in the actual calculation 
below, we indicate here how one can recover the set 
\(Y\) from the expression given below).
We choose
\( Q,R \subseteq [n] \) and \( \hat{Q}, \hat{R} \subseteq K \). Then we set
\[ S_1 = Q \cup \hat{Q} \ \ \ \ \ S_2 = R \cup \hat{R} \ \ \ \ \ Y = \hat{Q} \cup \hat{R}.\]
The important point to note here is that if \( x \in (R \cup Q) \cap K \) then we are
considering the event where \( x \not\in K^\prime \).  So the third arc out of 
\(x\) is present in the graph.  Furthermore, as we specify 
\( \hat{Q} \cup \hat{R} \subseteq K^\prime \), there is at least one arc pointing into each vertex in
\( \hat{Q} \cup \hat{R} \).

Now we explain the further division \( S_2 \).
We write
$\n=\n_0+\n_1+\n_2+\hat{\n}_0+\hat{\n}_1+\hat{\n}_2$ where $\n_i$
is the number of vertices $ x\in S_2 \setminus K^\prime $
for which $i$ of the arcs between \(x\) and \(\{ y_x, z_x \}\) are directed into \(x\) 
and \( \hat{\n}_i \) is the number of vertices $ x\in S_2 \cap K^\prime $
for which $i$ of the arcs between \(x\) and \( \{ y_x, z_x \} \) are directed into \(x\).
%
Let \( \rho = | S_1 \setminus K^\prime| \) and \( \hat{\rho}= |S_1 \cap K^\prime | \).
The probability that \( \G_D \) is in \( \cE_{\rm pieces} \) is at most
\begin{equation}
\label{exp}
\begin{split}
\sum_{s=n^{1/30}}^{ n/100 } & \sum_{\substack{\n_0,\n_1,\n_2\\\hat{\n}_0,\hat{\n}_1,\hat{\n}_2\\ 
\r , \hat{\rho}} }
\binom{n}{\r } \binom{k}{\hat{\r} } 
\binom{n}{\n_0,\n_1,\n_2} \binom{k}{\hat{\n}_0,\hat{\n}_1,\hat{\n}_2} 
\times
n^{\rho + \hat{\r} } \binom{n}{2}^{\n} \\ 
& \times
\brac{ 6
\cdot \frac{1}{n^2} \cdot
\frac{2s}{n}}^{\n_0}
\brac{ 18 \cdot \frac{1}{n^2} \cdot \bfrac{2s}{n}^2}^{\n_1}
\brac{ 9 \cdot \frac{1}{n^2} \cdot \left( \frac{2s}{n} \right)^3 }^{\n_2} \\
& \hskip1cm
\brac{ 2 \cdot \frac{1}{n^2} \cdot \frac{6s}{n} }^{ \hat{ \n}_0}
\brac{ 12 \cdot \frac{1}{n^2} \cdot
\frac{2s}{n}}^{\hat{\n}_1}
\brac{ 9 \cdot \frac{1}{n^2} \cdot
\bfrac{ 2s}{n}^2}^{\hat{\n}_2} \\ & \times 
\left( \frac{6}{n} \cdot \bfrac{2s}{n}^2 \right)^{\rho}
\left( \frac{6}{n} \cdot \frac{2s}{n} \right)^{\hat{\rho}}
\times\brac{1-\frac{s}{n}}^{3(n-3s-k)}.
\end{split}
\end{equation}
\begin{quote}
{\bf Explanation of (\ref{exp})} View the summand as five terms separated by the 
symbol \( \times \).  The first is an upper bound on the number of choices for \( S_1, S_2 \) and 
\( Y \).  
For such a fixed choice of \( S_1, S_2 \) and \( Y \) the rest of the 
expression in (\ref{exp}) is an upper bound on the probability of the union of all events of the
form \( \cF_{Y,S_1, S_2, T_1, T_2} \).
The
second term is an upper bound on the number of choices for \( T_1, T_2 \) and the explicit edges
between \( S_1 \cup S_2 \) and \( S_1 \cup T_1 \cup T_2 \) given by conditions (ii) and (iii).  Then we consider the
probability that the vertices in \( S_2 \) have the prescribed neighbors.  Note that if \( v \) is among
the \( \hat{\n_0} \) vertices in \( S_2 \cap K^\prime \) that has out-arcs directed to \( y_v \) and \( z_v \)
then, since \( v \in K^\prime \) and therefore \(v\) has in-degree at least 1, there is a vertex in \( T_1 \cup T_2 \)
that directs one of its out-arcs to \(v\).  The accounts for the `extra' \( 6s/n \) factor in this 
term of the product.
The fourth term is a bound on the probability that the arcs
out of \( S_1 \) are in the prescribed locations. 
Finally we multiply by the probability that
no vertex outside of \( S_1 \cup S_2 \cup T_1 \cup T_2 \) sends an arc into \( S_1 \cup S_2 \).
\end{quote}
%
%
The bound in (\ref{exp}) can be approximated from above by
\begin{align}
\sum_{s=n^{1/10}}^{n/100}\sum_{\substack{\n_0,\n_1,\n_2
\\\hat{\n}_0,\hat{\n}_1,\hat{\n}_2\\ \r, \hat{\r}}}
&
\bfrac{24 e s^2}{ \r n}^{\r} \bfrac{ 12 e k s}{ \hat{\r} n}^{\hat{\r}} \times
\bfrac{6es}{\n_0}^{\n_0}\bfrac{36es^2}{\n_1n }^{\n_1}
\bfrac{36es^3}{\n_2n^2}^{\n_2}\nonumber\\
&\bfrac{6esk}{\hat{\n}_0n}^{\hat{\n}_0}\bfrac{12esk}{\hat{\n}_1n}^{\hat{\n}_1}
\bfrac{18es^2k}{\hat{\n}_2n^2}^{\hat{\n}_2} \times e^{s ( - 3 + 9s/n + 3k/n)} .\label{vert}
\end{align}
We repeatedly apply the following observation (the proof of
which is given below):
\begin{equation}\label{ineq}
 x,y, \a,\b > 0 \ \ \ \Rightarrow \ \ \
\bfrac{\a}{x}^x\bfrac{\b}{y}^{y}\leq \bfrac{\a+\b}{x+y}^{x+y}.
\end{equation}
The expression in (\ref{vert}) can therefore be bounded above by
\begin{align}
&\sum_{s=n^{1/10}}^{ n/100}\sum_{\substack{\n_0,\n_1,\n_2\\\hat{\n}_0,\hat{\n}_1,\hat{\n}_2\\ \r, \hat{\r}}}
e^{s ( - 3 + 9s/n + 3k/n)} \times e^s 
\times \bfrac{ \ell(s, k,n)}{s}^s\nonumber
\end{align}
where
\begin{align*}
\ell(s,k,n) & = 24s (s/n) + 12s(k/n)+ 6s+36s (s/n)  + 36s(s/n)^2
+ 6s(k/n) \\
& \hskip3cm  +  12 s (k/n) +18 s (ks/n^2) \\
& = s \left( 6 + 60s/n + 36s^2/n^2 +  o(1) \right).
\end{align*}
%
Since we assume \( s \le n/100 \), we have \( \ell(s,k,n) \le 6.7 s \) and the
summand in (\ref{vert}) is at most
\[ \left[ e^{-1.91} 6.7 \right]^s \le (0.995)^s. \]
%
%
%
%
%
Thus the sum in (\ref{exp}) can be bounded by
$n^8(.995)^{ n^{1/30}}$, and this proves Lemma~\ref{lem:pieces}.

\begin{proof}[Proof of (\ref{ineq})]
Let $f(x)=\bfrac{\a}{x}^x\bfrac{\b}{z-x}^{z-x}$. Then we have
\begin{eqnarray*}
\frac{f'(x)}{f(x)}&=&\log\a-1-\log x-\log\b+1+\log(z-x)\\
&=&\log\bfrac{(z-x)\a}{\b x}.
\end{eqnarray*}
So, \( f^\prime(x) = 0 \) iff $x=x^*=\frac{z\a}{\a+\b}$.

Differentiating once more, we get
$$\frac{f''(x)}{f(x)}-\frac{f'(x)^2}{f(x)^2}=-\frac{1}{x}-\frac{1}{z-x}<0$$
and so $f''(x^*)<0$ and $x^*$ is the maximum over \( x \in [0,z] \).
\end{proof}



\section{Conditioning: Proof of Lemma~\ref{lem:uni}}
\label{sec:uni}

We begin with some definitions and initial observations, including the definition
of the event \( \cE_D \).
Set 
\[ 
\t = 3 - \frac{1}{ \log ^{1/2} n } = 3-\a  \ \ \ \ \ \ \ \ \ \ \
\ell = \rdup{10 \log_\t \log n}.\]  

Recall that \( \cX \) is a fixed part in some partition in the filtration of
\( \Omega_k \times [n]^k \) given by the execution of the 
Extension-Rotation Procedure.  Here we 
work with the space of digraphs \( \Omega_k \); the edges 
that have been added to
raise the minimum degree to 3 or in the course of the 
Extension-Rotation Procedure play no role.

For \(x \in [n] \setminus (I_\cX \cup O_\cX ) \)
let \( \cD_1(\cX)\) be partitioned into 
\( \cD_x, \overline{\cD_x} \) 
where \( D \in \cD_x \) if
\( x \in K(D) \).  Note that for fixed 
\(x \) we have \( \Pr(x \in K(D)) = | \cD_x|/ |\cD_1(\cX)| \).

Define a bipartite graph \( \Sigma = \Sigma_x \) on vertex set \( \cD_1( \cX) \) by including
an edge joining \( D \in \cD_x \) and \( D^\prime \in \overline{\cD_x} \) if \( D^\prime \) is 
obtained by reversing the arcs along a shortest path of length \( \ell \) in \(D\) that starts
at a vertex in \( V \setminus ( K(D) \cup O_\cX ) \) and ends at \(x\).  
We have, where $d_\S$ denotes degree in $\S$,
\begin{equation}
\label{eq:sum}
\sum_{ D \in \cD_x} d_\S(D) = \sum_{ D \in \overline{ \cD_x}} d_\S(D). 
\end{equation}
We get a lower bound on 
\( |\cD_x|/ |\cD_1(\cX)| \) by bounding the degrees in \( \Sigma \): We establish upper bounds on 
\( d_\Sigma(D) \) for \( D \in \cD_x \) and lower bounds on \( d_\Sigma(D) \) 
for \( D \in \overline{ \cD_x} \).

For a digraph \(D \in \Omega_k\) let \( \Pi_{D, x \from V } \) be 
the set of vertices
\(y \in V \) such that there exists a shortest path of length 
\( \ell \) from \( y \) to \(x\).  
Note that $\t$ is the average out-degree in $D\in \Omega_k$ and so $\t^\ell\approx \log^{10}n$ is
approximately equal to the expected number of vertices reachable by a path of length $\ell$.
Let 
\(  \Pi_{D, x \to K } \) be the set of
vertices \(z \in K \) such that there is a shortest path 
from \( x \) to \(z\) of length \( \ell \).  We bound \( |\Pi_{D,x \from V }| \) and 
\( |\Pi_{D,x \to K}| \) 
for most \( D \in \cD_1(\cX) \) and \( x \in [n] \).  
To this end, we define the following cut-offs:
\begin{gather*}
 \pi_{\rm in} = ( 1000 \log \log n ) \log^{10} n
 \ \ \ \ \ \ \ \ \ \pi_{\rm out } = \frac{  \log^{10} n}{ (\log \log n)^2 } \cdot \frac{k}{n}
 \\
\r_{\rm in } = (\log n)^{20} 
.
\end{gather*}
Let \( \Psi_{\from V} \subseteq \Omega_k \) be the event 
$$\card{\set{ x \in [n] : |\Pi_{D, x \from V}| \ge \pi_{\rm in} }}
\ge \frac{ 10 n}{ (\log n)^{100}} \text{ or } 
\exists x \in [n] \text{ such that }   |\Pi_{D, x \from V}| \ge \r_{\rm in}.$$
Similarly, define \( \Psi_{\to K} \subseteq \Omega_k  \) to be the event
\begin{gather*}
\left| \left\{ x \in [n] : |\Pi_{D, x \to K}| \le \pi_{\rm out} \right\} \right| 
\ge \frac{ n}{ (\log n)^{100}}.
\end{gather*}  
We now define 
\begin{equation}\label{ED} 
\cE_D = \Psi_{\to K} \cup \Psi_{\from V}.
\end{equation}

%
%
\begin{lemma}
\label{lem:structs}
If \( D \) is chosen uniformly at random from \( \Omega_k \) then
\begin{equation*}
\Pr( D \in \cE_D  \  
\wedge \    \G_D \not\in \cE_G ) \le e^{ - \frac{1}{4} \log ^3 n } .
\end{equation*}
\end{lemma}
\noindent The proof of Lemma~\ref{lem:structs} is given in Section~\ref{sec:sub}.   Note that condition 
(\ref{eq:digraph}) 
follows from Lemma~\ref{lem:structs}.  We now proceed with the proof of Lemma~\ref{lem:uni}.

For each \(D \in \cD_1(\cX)\) we define three sets of vertices:
\begin{align*}
V_1 &= V_1(D) = \left\{ x \in [n]:  |\Pi_{D,x \from V}| \ge \pi_{\rm in} \right\} \\
V_2 &= V_2(D) = \left\{ x \in [n]:  |\Pi_{D,x \to K}| \le \pi_{\rm out} \right\} \\
V_3 &= V_3(D) = \left\{ x \in [n]:  \left| \Pi_{D, x \to K} \cap I_{\cX} \right| \ge \pi_{\rm out}/2 \right\}
\end{align*}
The key observation is that most 
\(x \in [n] \) have the property that 
\( x \in V_i(D) \) for very few \( D \in \cD_1(\cX)\) for \( i =1, 2,3 \).  Indeed, the number of ordered pairs
\( (x,D) \) where \(x \in [n] \), \( D \in \cD_1(\cX) \setminus \Psi_{ \from V} \) 
and \( x \in V_1(D) \) is at most
\( 10| \cD_1(\cX) | n / ( \log^{100} n) \).  It follows that
\begin{equation}
\label{eq:v1}
\left| \left\{ x \in [n]: \left|  \left\{ D \in \cD_1(\cX) \setminus \Psi_{\from V} : 
x \in V_1(D) \right\} \right| \ge \frac{|\cD_1(\cX)|}{ \log^{99} n} \right\} \right|
\le  \frac{10n}{ \log n}. 
\end{equation}
Similarly, we have
\begin{equation}
\label{eq:v2}
\left| \left\{ x \in [n]: \left|  \left\{ D \in \cD_1(\cX) \setminus \Psi_{\to K} : 
x \in V_2(D) \right\} \right| \ge \frac{|\cD_1(\cX)|}{ \log^{99} n} \right\} \right|
\le  \frac{n}{ \log n}. 
\end{equation}
Now we turn to the set of vertices \(x \in [n]\) with the 
property that \(x \in V_3(D) \) for
many digraphs \(D\).  Consider a fixed digraph 
\(D \in \cD_1(\cX) \setminus \Psi_{\from V} \).  The number of ordered pairs 
\( (a,b) \) where \(a \in V \),
\(b \in I_\cX \) and there is a shortest path of length \( \ell \) 
from \( a \) to \( b \) is at most
\[
\frac{10n}{ \log^{100} n} \cdot \r_{\rm in} 
+ | I_\cX | \cdot \pi_{\rm in}
\le \frac{ n \log^{10} n}{ \log n } \cdot ( \log \log n )^2  
\]
for \(n\) sufficiently large.
It follows that \( |V_3(D) | \le 2 n ( \log \log n)^4 
/( \log^{1/2} n ) \).  Therefore
\begin{equation}
\label{eq:v3}
\left| \left\{x \in [n] : \left|  \left\{ D \in \cD_1(\cX) \setminus \Psi_{\from V} : 
x \in V_3(D) \right\} \right| \ge \frac{|\cD_1(\cX)|}{ \log^{1/5} n} \right\} \right|
\le  \frac{n}{ \log^{1/5} n}. 
\end{equation}
To prove
Lemma~\ref{lem:uni} it suffices to show
\begin{equation*}
\left| \left\{ x \in [n] \setminus (I_\cX \cup O_\cX ) :  
\frac{ |\cD_x| }{ | \cD_1(\cX) |} \le \frac{1}{  4000 (\log \log n)^3 } \cdot \frac{k}{n}  \right\} \right| 
= o(n).
\end{equation*}
We may therefore restrict our attention to vertices \(x\) in the complements of the sets 
in equations (\ref{eq:v1}), (\ref{eq:v2}) and (\ref{eq:v3}).

Now apply (\ref{eq:sum}).  We have
\begin{multline*}
| \cD_x| \pi_{\rm in} + \frac{ |\cD_1(\cX)|}{ \log^{99} n } \r_{\rm in} + 
\frac{ n^2 |\cD_1(\cX)|}{ e^{ \log^2n}} \D_0^\ell \ge \sum_{D \in \cD_x} d_\Sigma(D) \\
= \sum_{D \in \overline{ \cD_x}}
d_\Sigma(D) 
\ge \left( |\cD_1(\cX)| - |\cD_x| - \frac{ | \cD_1(\cX)|}{ \log^{99} n} - \frac{ |\cD_1(\cX)|}{ \log^{1/5} n} \right)
\frac{\pi_{\rm out}}{2}
\end{multline*}
Therefore
\[ | \cD_x | \left( \pi_{\rm in} + \frac{\pi_{\rm out}}{2} \right) \ge
| \cD_1(\cX) |  \frac{ \pi_{\rm out}}{2} 
\left( 1 - \frac{1}{ \log^{99} n} - \frac{1}{ \log^{1/5}n} - 
\frac{ 2 \r_{\rm in}}{ \pi_{\rm out} \log^{99}n} - \frac{ 2n^2 \D_0^\ell}{ \pi_{\rm out} e^{\log^2 n}} \right) \]
and
\[ \frac{ | \cD_x|}{ | \cD_1(\cX)| } \ge \frac{ \pi_{\rm out}}{ 4 \pi_{\rm in}} = \frac{1}{  4000 ( \log \log n)^3 } 
\cdot \frac{k}{n}. \]

%
%
%
%


\subsection{Proof of Lemma~\ref{lem:structs}}
\label{sec:sub}

We begin with the event \( \Psi_{\from V} \).  Let \(x\) be a fixed vertex.
For \(i =1, \dots, \ell \) let
\( S_i = S_i(x) \) be the set of vertices
that can reach \(x\) in a path of length \( i \) but 
not via a shorter path.  We reveal the sets
\(S_i \) iteratively (i.e. using breadth first search).  

\begin{claim} 
\label{cl:ain}
If \( f(n) \) is an arbitrary positive real function then
\[ Pr\left( \left|S_{\ell} \right| \ge  f(n) \tau^{\ell} \wedge \G_D \not\in \cE_G \right) 
\le e^{ 1/2-f(n)/10}  \]
for \(n\) sufficiently large.
\end{claim}
\begin{proof}
Note that \( \G_D \not\in \cE_G \) implies $\D(G)\leq \D_0 = \frac{ 2 \log n}{ \log \log n} $ and so
\[ \sum_{j=0}^{\ell} |S_i| \le \sum_{j=0}^{\ell} \D_0^j \le (\ell+1) \D_0^{\ell} 
= n^{o(1)}. \]
Thus, we may discard the set of digraphs for which there exists an index \(1 \le i \le \ell\) such that
\( |S_i| > \sqrt{n} / \ell \).  In other words,
we iteratively condition on the event \( |S_i| \le \sqrt{n} / \ell \).

Now an observation.  If \( 0 < \l < 1 \) and \( \sum_{j=0}^i |S_j| \le \sqrt{n} \) then 
we have
\begin{equation*}
\begin{split}
E[ e^{\l |S_{i+1}|} \mid  |S_i| ] & \le \sum_{m=0}^{\t n} \binom{\t n}{m} 
\bfrac{ |S_i|}{ n -  \sqrt{n}}^m \left(1 - \frac{ |S_i|}{n - \sqrt{n}}\right)^{\t n-m} e^{\l m} \\
& = \left(1 - \frac{ |S_i|}{ n - \sqrt{n}} + \frac{ |S_i| e^\l}{ n - \sqrt{n}} \right)^{\t n} \\
& \le \left( 1 + \l(1+\l) \frac{ |S_i|}{ n - \sqrt{n}} \right)^{\t n} \\
& \le e^{ \t |S_i| \l( 1 + \l)( 1 + 2/\sqrt{n})}.
\end{split}
\end{equation*}
Now set \( \l_0 = 1/2 \).  
For \( i =1, \dots, \ell \) we define \( \l_i \) by setting
\[ \l_{i-1} = \t \l_i ( 1 + \l_i)( 1 + 2 / \sqrt{n}). \]
Note that so long as \( \l_i > \t/ n^{1/4} \) then we have
\[ \frac{ \l_i}{\t} \left( 1 - \frac{ \l_i}{\t} \right) \le \l_{i+1} \le   \frac{ \l_i}{\t}.   \]
Then inductively, so long as \( \l_j > \t/ n^{1/4} \), we have \( \l_j \le 1/( 2 \t^j) \) and 
\begin{multline*}
\l_j \ge \frac{ \l_0}{t^j} \prod_{m=0}^{j-1} \left( 1 - \frac{ \l_m}{\t} \right) \ge 
\frac{ \l_0}{\t^j} \prod_{m=0}^{j-1} \left( 1 - \frac{1}{2 \t^{m+1}} \right) \\
\ge  \frac{ \l_0}{\t^j} \exp \left\{ 1 - \sum_{m=0}^{j-1} \frac{1}{ \t^{m+1}} \right\}
\ge \frac{ \l_0}{\t^j} \exp \left\{ \frac{ \t}{ \t-1} \right\}
\ge \frac{1}{10 \t^j}.
\end{multline*}
(Note that we use the inequality \( 1- x > e^{-2x} \) for \( x \le 1/2 \).)  In particular, we have
\begin{equation}
\label{eq:lam}
\l_{\ell} \ge \frac{1}{10 \t^{\ell}}.
\end{equation}
We now define a random variable \(Y_i\). If \( |S_1|, |S_2|, \dots, |S_{i-1}| \le \sqrt{n}/\ell \)
then set \( Y_i = |S_i| \).  If, on the other hand, there exists \(j < i \) such that 
\( |S_{j}| > \sqrt{n}/\ell \) then set \( Y_i = \l_{i-1} Y_{i-1}/ \l_i \).  Let \( \cB_i \)
be the event that there exists \(j < i \) such that 
\( |S_{j}| > \sqrt{n}/\ell \).
For \( i =1, \dots, \ell \) we have
\begin{multline*}
E\left[ e^{\l_i Y_i} \mid \overline{ \cB_i} \right] = E \left[ E \left[ e^{\l_i |S_i|} \mid |S_{i-1}| \right] 
\mid \overline{\cB_i} \right] \\
\le E \left[ e^{\t |S_{i-1}| \l_i(1- \l_i)( 1 + 2/\sqrt{n})}  \mid \overline{\cB_i} \right]
= E \left[ e^{ \l_{i-1} |S_{i-1}|} \mid \overline{\cB_i} \right] 
= E \left[ e^{ \l_{i-1} Y_{i-1}} \mid \overline{\cB_i} \right]. 
\end{multline*}
Of course, the same inequality holds if we condition instead on \( \cB_i \).
It follows that
\( E \left[ e^{ \l_{\ell} Y_\ell } \right] \le E \left[ e^{ \l_0 Y_0 } \right] = e^{1/2} \).  Applying 
Markov's inequality and (\ref{eq:lam}) we have
\[ Pr\left( Y_\ell \ge f(n) \tau^{\ell}  \right) 
\le E \left[ e^{\l_{\ell} Y_\ell} \right] e^{ - \l_{\ell} f(n) \t^{\ell} } 
\le e^{\frac{1}{2} -  \frac{1}{10} f(n)}. \]
As the event \( |S_\ell| \neq Y_\ell \) is a subset of \( \G_D \in \cE_G \), the claim follows.
\end{proof}

Let \(X\) be the set of vertices \(x\) such that 
\( | \Pi_{D, x \from V}| = 
|S_{\ell}(x)| \ge 1000 \log \log n \cdot \log^{10} n \).  Applying Claim~\ref{cl:ain} 
we have
\begin{equation*}
\begin{split}
\Pr(x \in X \wedge \G_D \not\in \cE_G) 
& \le \frac{e^{1/2}}{ \log^{100} n}.
\end{split}
\end{equation*}
Therefore, \( E[|X| \mid \G_D \not\in \cE_G ] \le 2n/ \log^{100} n \) and \( E[|X|] \le 3n/ \log^{100}n\) (as
\( \Pr(\G_D \in \cE_G) = o(1/\sqrt{n}) \) and \( |X| \le n \)) .  We bound the probability that \(X\) 
is significantly larger than its expected value by applying the Azuma-Hoeffding inequality to the
arc-exposure martingale on $\Omega_K$. (This is an instance of the first general setting in Section 7.4 of Alon and Spencer
\cite{AS}.) We fix $K$ and expose the out-arcs one at a time.  Let \( a_i \) be the \( i^{\rm th} \)
arc in some ordering of the $3n-k$ out-arcs.  Set
\( X_i = E[ |X| \mid a_1, a_2, \dots, a_i ] \).
Note that if \(D\) and \( D^\prime \) differ by a single out-arc then we have
\[  \left| |X(D)| - |X(D^\prime)| \right| < \sum_{i=0}^\ell 3^{i} \le \log^{11} n. \]
It follows that \( |X_{i+1} - X_i| \le \log^{11} n \) for \( i \ge 0 \).  The Azuma-Hoeffding inequality then
implies 
\[ Pr \left( X \ge \frac{10 n}{ \log^{100}n} \right) 
\le Pr\left( X \ge E[X] + \frac{n}{ \log^{100} n } \right) \le e^{ -n^{1-o(1)} } . \]

To bound the probability that there exists a vertex x such that \( |\Pi_{D,x\from V}| \ge \r_{\rm in} \) 
we apply Claim~\ref{cl:ain} with \( f(n) = \log^{10} n\).
It follows that
\( \Pr( |S_\ell| \ge \r_{\rm in} \wedge \G_D \not\in \cE_G ) \le e^{ -(\log^{10} n)/10} \).

Now we turn to the event \( \Psi_{\to K} \). Let \(y\) be a fixed vertex.
Let \( T_i = T_i(y) \) 
be the sum of the out-degrees of the set of 
vertices that can be reached from \( y \) via directed paths of length \(i\) but not 
via shorter paths (i.e. \( T_0 = 2\text{ or }3 \)).  As above, we reveal \(T_i\) iteratively.  
Recall that \( \ell = \rdup{10 \log_\t \log n} \).  Set
\( \ell_1 = \ell -1 - \log_2( 1000 \log \log n) \).
\begin{claim}
\label{cl:out}
\[ Pr\left( T_{\ell-1} \le \t^{\ell_1}  \right)
\le \frac{1}{ \log^{2000} n }. \]
\end{claim}
\begin{proof}
We begin with an observation.
Define \( \g_i = 3^{i+2}/(2n) \).  Note that the probability that an out-arc from \(T_i\) hits a vertex 
that has already been seen is bounded above by \( \g_i \).  Then
\begin{equation}
\label{eq:base}
E\left[ e^{ -\l |T_{i+1}|} \mid |T_i| \right] \le  
\left( \g_i + \a e^{-2 \l} + ( 1 - \a - \g_i) e^{-3 \l} \right)^{|T_i|}. 
\end{equation}
We apply (\ref{eq:base}) in two ways.  First, if \( \l > 0 \) and 
\( \l e^{-3\l} \ge 2 \g_i \) then on replacing the first $\g_i$ by $\l e^{-3\l}/2$ and dropping the second $\g_i$ we see that
\[ \g_i + \a e^{-2 \l} + (1- \a - \g_i) e^{-3\l} \le e^{-2 \l} \]
and so
\begin{equation*}
\begin{split}
E\left[ e^{ -\l |T_{i+1}|} \mid |T_i| \right] \le e^{ -2 \l |T_i|} .
\end{split}
\end{equation*}
Furthermore, if \( 3 \g_i \le \l \) then we have
\begin{equation*}
\begin{split}
E \left[ e^{ -\l |T_{i+1}|} \mid |T_i| \right] 
& \le  
\left( \g_i + \a(1 - 2\l + 3\l^2) + ( 1 - \a - \g_i) (1 -3\l + 5\l^2) \right)^{|T_i|} \\
& \le ( 1 - \l (\t - 6\l))^{|T_i|}\\
& \le e^{ -\l(\t - 6 \l)|T_i|} .
\end{split}
\end{equation*}
Set \( \l_0 = 1000 \log\log n \) and \( i_0 = \ell - \ell_1 +3 \).  
For \( i =1 ,\dots, \ell-1 \) define \( \l_i \) by
\[ \l_{i-1} = \begin{cases}
2 \l_i & \text{ if } i \le i_0 \\
\l_i( \t - 6\l_i) & \text{ if } i >  i_0
\end{cases} \]
(with the additional condition \( \l_i < \t/12 \) in the latter case).
Of course, we make these choices in order to establish
\(  E[ e^{-\l_{i+1} |T_{i+1}|} ] \le E[ e ^{ - \l_i |T_i|} ] \) for \( i =0, \dots, \ell-1\) and 
therefore
\[ E[ e^{-\l_{\ell-1} |T_{\ell-1}|} ] \le E [ e^{-\l_0 |T_0|} ] = \frac{1}{ \log^{3000} n}. \]
Applying Markov's inequality we have that for any $s>0$,
\[ Pr[ |T_{\ell-1}| \le s \t^{\ell-1}] \le  e^{\l_{\ell-1} s \t^{\ell-1}} E [ e^{-\l_{\ell-1} |T_{\ell-1}|} ] 
\le \frac{ e^{\l_{\ell-1} s \t^{\ell-1}}}{ \log^{3000}n}. \]
We now establish an upper bound on \( \l_{\ell-1} \).  Note first
that $\l_{i_0}=\frac{1}{16}$ and $\l_{i_0+1}\leq \frac{1}{32}$. Then
observe that $i>i_0$ and $\l_i\leq \frac1{24}$ implies 
\begin{equation}
\label{eq:upper}
  \l_i \le \frac{\l_{i-1}}{\t} \left( 1 + \frac{ 8 \l_{i-1}}{\t} \right). 
\end{equation}
Furthermore, we have \( \l_{i} \le \l_{i-1}/2 \) for all \(i\) and the bound
\[ \prod_{j} \left( 1 + \frac{8 \l_j}{ \t} \right) \le 
\exp \left\{ \sum_j \frac{8 \l_j}{ \t} \right\} \le e^{16/\t} \le e^{6}, \]
then implies
\[ \l_{\ell-1} \le 2^{-4} \t^{ -\ell_1 + 4} e^6. \]
Taking \(s = \t^{- \ell-1 + \ell_1} 
\), 
we have
\[ Pr \left( |T_{\ell-1}| \le \tau^{\ell_1} \right) \le 
\frac{ e^{ e^6 (\t/2)^{4}}}{ \log^{3000}n}.\]   
\end{proof}

Now let \(Y\) be the set of vertices \(y\) such that 
\( T_{\ell-1}(y) \le \t^{\ell_1} \).  It follows from Claim~\ref{cl:out} that 
we have \( E[|Y|] \le n/ \log^{2000} n \).  We again apply the Azuma-Hoeffding
inequality to the arc exposure martingale on $\Omega_K$ to show that the \(|Y|\) is 
concentrated around its expected value.  Recall that \( a_i \) denotes the \( i^{\rm th} \)
arc observed in this process.  Set
\[ Y_i = \begin{cases}
E[ |Y| \mid a_1, a_2, \dots, a_i ] & \text{ if }  Pr\left( \D( \G_D) \ge \log^3 n \mid  a_1, \dots a_{i-1} \right) 
< \frac{1}{n^2} \\
Y_{i-1} & \text{ if }\exists j<i:   Pr\left( \D( \G_D) \ge \log^3 n  \mid  a_1, \dots a_j \right) \ge \frac{1}{n^2}
. 
\end{cases}
\]
In order to bound the one step changes in this sequence of random variables, we can restrict our attention to 
the situation where $Pr\left( \D( \G_D) \ge \log^3 n \mid  a_1, \dots a_{i-1} \right) 
< \frac{1}{n^2}$. First note that if \(D\) and \( D^\prime \) differ by a single out-arc 
and \( \D(\G_D), \D( \G_{D^\prime} ) < \log^3 n\) 
then we have
\[  \left| |Y(D)| - |Y(D^\prime)| \right| < (\log^3 n)^{\ell-1} < n^{1/10}. \]
To deal with pairs $D,D'$ where \( \D(\G_D)\) or \( \D( \G_{D^\prime} ) \geq \log^3 n\) we note 
\[ Pr \left(   \D( \G_D) \ge \log^3 n \mid  a_1, \dots a_{i} \right) < n  
Pr\left(  \D( \G_D) \ge \log^3 n \mid  a_1, \dots a_{i-1} \right)<1/n \] 
and \( |Y| \le n \).  Putting it all together we have
\( |Y_{i+1} - Y_i| \le n^{1/10}+n/n\le 2 n^{1/10} \).  
The Azuma-Hoeffding inequality then implies
\[ Pr\left( Y_{3n-k} \ge E[|Y|] + \frac{n}{ \log^{200} n} \right) \le e^{ - n^{8/10-o(1)}}. \]
Noting that
\begin{equation*}
\begin{split}
\Pr( Y_{3n-k} \neq |Y|) & \le n^2 \Pr( \D( \G_D) > \log^3 n) \\
& < n^2 \cdot n \binom{ 3n}{ \log^3n -3} 
\bfrac{1}{n}^{\log^3n -3} \\ & \le e^{-\log^3 n}, 
\end{split}
\end{equation*}
we conclude that
\begin{equation}\label{hhh}
Pr \left( |Y| \ge n/ ( \log^{150} n) \right) \le e^{- (\log^3 n)/2}.
\end{equation}
Suppose now that $x\notin Y$. Then 
\begin{equation}\label{hh}
\Pi_{D,x\to K}\ dominates\  Bi(\t^{\ell_1},k/n)-Bi(3^{\ell},3^{\ell}/n).
\end{equation}
The first binomial in \eqref{hh} is the number of edges from $T_{\ell-1}(y)$ to $K$ and the second
binomial dominates the number of vertices chosen twice by these edges. The lemma now follows from
\eqref{hhh}, \eqref{hh} and simple bounds on the binomials:
$$\Pr(Bi(\t^{\ell_1},k/n)\leq \t^{\ell_1}k/2n)\leq e^{-\Omega(\log^9n)};\; \Pr(Bi(3^{\ell},3^{\ell}/n)\geq \log^3n)\leq e^{-\Omega(\log^4n)}.$$
\proofend

\vskip5mm

\noindent
{\bf Final Remark.}  We have established that \( G_{\rm 3-out} \) is 
Hamiltonian with high probability.  In general, it may be the case if \( m \ge 3 \) then
with high probability \( G_{m-{\rm out}} \) contains \( \lfloor m/2 \rfloor \) disjoint
Hamilton cycles plus, in the case that $m$ is odd, a disjoint matching that saturates all but at most 
one vertex.  We leave this as an open problem.

\vskip5mm
{\bf Acknowledgment.}  The authors thank the anonymous referees for many useful comments.


\end{document}